\documentclass[a4paper,notitlepage,twoside,leqno,11pt]{amsart}

\usepackage{bbm,pifont,latexsym}
\usepackage{anysize}
\marginsize{2.8cm}{2.8cm}{2.5cm}{2.5cm}

\usepackage{dcolumn,indentfirst}
\usepackage[colorlinks=false,linkcolor=black,citecolor=black,urlcolor=black,bookmarksopen=true,linktocpage=true,pdfstartview={XYZ null null 1.25}]{hyperref}
\usepackage{amsmath,amssymb,amscd,amsthm,amsfonts,mathrsfs}
\usepackage{color,graphicx,xcolor,graphics,subfigure,extarrows,caption2}
\usepackage{titletoc}

\usepackage{thm-restate}

\newtheorem{thm}{Theorem}[section]
\newtheorem{lema}[thm]{Lemma}
\newtheorem{cor}[thm]{Corollary}
\newtheorem{prop}[thm]{Proposition}

\newtheorem*{mainthm}{Main Theorem}

\theoremstyle{definition}

\newcommand{\EC}{\widehat{\mathbb{C}}}
\newcommand{\A}{\mathbb{A}}
\newcommand{\C}{\mathbb{C}}

\newcommand{\N}{\mathbb{N}}

\newcommand{\T}{\mathbb{T}}

\newcommand{\ii}{\textup{i}}

\newcommand{\Crit}{\textup{Crit}}

\newcommand{\Int}{\textup{int}}
\newcommand{\Ext}{\textup{ext}}

\newcommand{\subt}{\smallsetminus}

\usepackage{enumerate,enumitem}

\begin{document}

\author{FEI YANG}
\address{Fei Yang, Department of Mathematics, Nanjing University, Nanjing, 210093, P. R. China}
\email{yangfei\rule[-2pt]{0.2cm}{0.5pt}math@163.com}

\title{RATIONAL MAPS WITHOUT HERMAN RINGS}

\begin{abstract}
Let $f$ be a rational map with degree at least two. We prove that $f$ has at least $2$ disjoint and infinite critical orbits in the Julia set if it has a Herman ring.
This result is sharp in the following sense: there exists a cubic rational map having exactly two critical grand orbits but also having a Herman ring. In particular, $f$ has no Herman rings if it has at most one infinite critical orbit in the Julia set. These criterions derive some known results about the rational maps without Herman rings.
\end{abstract}



\date{\today}



\maketitle

{\setcounter{tocdepth}{1}
\tableofcontents
}

\section{Introduction}

Let $f$ be a rational map with degree at least two. The \emph{Fatou set} $F(f)$ of $f$ is defined as the maximal open subset on the Riemann sphere $\EC$ in which the sequence of iterates $\{f^{\circ n}\}_{n\geq 0}$ is normal in the sense of Montel. The \emph{Julia set} $J(f):=\EC\subt F(f)$ of $f$ is the complement of the Fatou set in $\EC$. Each connected component of the Fatou set is called a \emph{Fatou component}. A Fatou component $U$ is called \emph{periodic} if there exists a positive integer $n$ such that $f^{\circ n}(U)=U$. According to Sullivan \cite{Sul85}, there are five types of periodic Fatou components: attracting basin, superattracting basin, parabolic basin, Siegel disk and Herman ring. Moreover, all the Fatou components are iterated to one of these five types periodic Fatou components eventually. A periodic Fatou component $U$ is called a \textit{Herman ring} if $U$ is conformally isomorphic to some annulus and if the restriction of $f$, or some iterate of $f$ on $U$, is holomorphically conjugated to an irrational rotation of this annulus.

According to K{\oe}nigs, B\"{o}ttcher, Leau-Fatou and Siegel (see \cite[\S8-\S11]{Mil06}), it is relatively easy to construct a family of rational maps such that each one of them has an attracting basin, superattracting basin, parabolic basin or Siegel disk with fixed multiplier since these four types of periodic Fatou components have an associate periodic point (attracting, superattracting, rational neutral or irrational neutral). However, for the case of Herman ring, because there exists no associate periodic point, it is difficult to judge whether a given rational map (or family) has a Herman ring or not. This is not strange: Herman ring is the last type of periodic Fatou component that has been found. There are two known methods for constructing Herman rings. The original method is based on a careful analysis on the analytic diffeomorphism of the circle, due to Herman \cite{Her79}. An alternative method is based on quasiconformal surgery, by pasting together two Siegel disks, due to Shishikura \cite{Shi87}.

In \cite[\S 2.4]{Lyu86}, Lyubich asked the following question: Does there exist a parameter $\omega\in\C\subt\{0\}$ such that $f_\omega(z)=1+\omega/z^2$ has a Herman ring? In his master's thesis \cite[Theorem 3]{Shi87}, Shishikura proved that any quadratic rational map has no Herman rings and hence answered Lyubich's question. The method of Shishikura is analytical and profound. A completely topological method of proving the non-existence of Herman rings of
\begin{equation}
f_\omega(z)=1+\omega/z^d,
\end{equation}
where $d\geq 2$ and $\omega\in\C\subt\{0\}$, was established by Bam\'{o}n and Bobenrieth \cite{BB99}.

To study the connectivity of the Julia sets of rational maps, one often needs to exclude the existence of Herman rings.
The dynamics of McMullen maps \begin{equation}\label{eq-McMullen}
F_\lambda(z)=z^m+\lambda/z^\ell,
\end{equation}
where $m,\ell\geq 2$ and $\lambda\in\C\subt\{0\}$, has been studied extensively recently (see \cite{DLU05,Ste06,QWY12} and the references therein). By a careful analysis of the dynamical symmetries on $F_\lambda$, Xiao and Qiu proved that $F_\lambda$ has no Herman rings \cite{XQ10}. A much simpler argument of the non-existence of Herman rings of the McMullen maps appeared in \cite{DR13}.

\subsection{Statement of the results}

In this article, we give a general criterion to exclude the existence of Herman rings. Let $f$ be a rational map and $z$ a point on $\EC$. The \emph{forward orbit} and \emph{grand orbit} of $z$ under $f$ are defined as $O_f(z):=\{f^{\circ n}(z)\}_{n\in\N}$ and $GO_f(z):=\{f^{-m}\circ f^{\circ n}(z)\}_{m,n\in\N}$ respectively. A forward orbit $O_f(z)$ is called \textit{critical} if $z$ is a critical point of $f$. The dynamics of a rational map is determined by the critical orbits $\{O_f(c):c\in\Crit(f)\}$, where $\Crit(f)$ is the set of critical points of $f$. It was known that every boundary point of a Herman ring belongs to the closure of the orbit of some critical point (see \cite[Lemma 15.7]{Mil06}). Therefore, one can believe that a rational map has no Herman rings provided it has relatively simple critical orbits. The orbit $O_f(z)$ is called \emph{infinite} if the cardinal number of $O_f(z)$ is equal to $+\infty$, i.e. $f^{\circ m}(z)\neq f^{\circ n}(z)$ if $m\neq n$. In this article, we prove the following Main Theorem.

\begin{mainthm}
Let $f$ be a rational map with degree at least two which has a Herman ring. Then it has at least $2$ disjoint and infinite critical orbits in the Julia set.
\end{mainthm}

In Appendix A of \cite{Mil00}, Milnor generalized Shishikura's result and proved that any rational map with only two critical points (without counting multiplicity) cannot have any Herman rings. Although some similar criterions as stated in the Main Theorem may be known for the experts in this field, we would like to mention that there are no restrictions on the number of critical points in our criterion. In particular, the Main Theorem is sharp in the following sense:

\begin{restatable}{thm}{thmsharpintro}
\label{thm-sharp}
There exist suitable parameters $a$, $b$ and $t$ such that the cubic rational map
\begin{equation}\label{equ-f-a-b}
f_{a,b}(z)=e^{2\pi\ii t}\frac{z-a}{1-\overline{a}z}\left(\frac{z-b}{1-\overline{b}z}\right)^2, \text{ where } |a|>1, |b|<1 \text{ and }t\in(0,1),
\end{equation}
has exactly two disjoint critical grand orbits and has also a Herman ring.
\end{restatable}

The idea of the construction in Theorem \ref{thm-sharp} is as following:  Note that $b$ is a critical point of $f_{a,b}$. We find a suitable $a$ such that the origin is also a critical point of $f_{a,b}$. Then $f_{a,b}$ maps the critical point $b$ to another critical point $0$. The dynamical symmetry guarantees that $f_{a,b}$ maps the critical point $1/\overline{b}$ to another critical point $\infty$. The argument principle indicates that the restriction of $f_{a,b}$ on the unit circle is a real analytic diffeomorphism. For any given Diophantine irrational number $\theta$, we can find a suitable $t$ such that the rotation number of the restriction of $f_{a,b}$ on the unit circle is $\theta$ and hence $f_{a,b}$ has a fixed Herman ring surrounding the unit circle. See Figure \ref{Fig_Julia-Herman-sphere} for an example.

\begin{figure}[!htpb]
  \setlength{\unitlength}{1mm}
  \includegraphics[width=70mm]{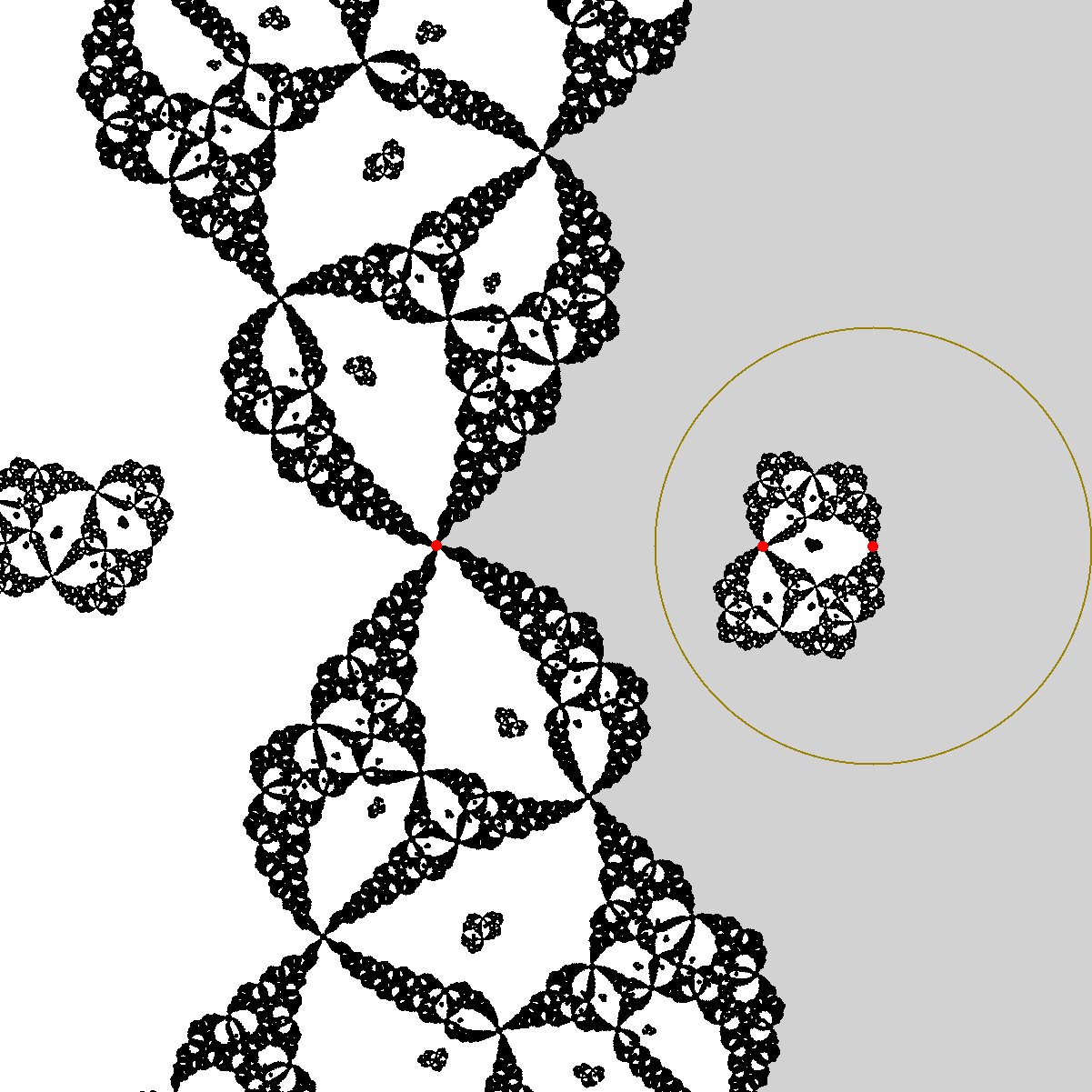}\quad
  \includegraphics[width=70mm]{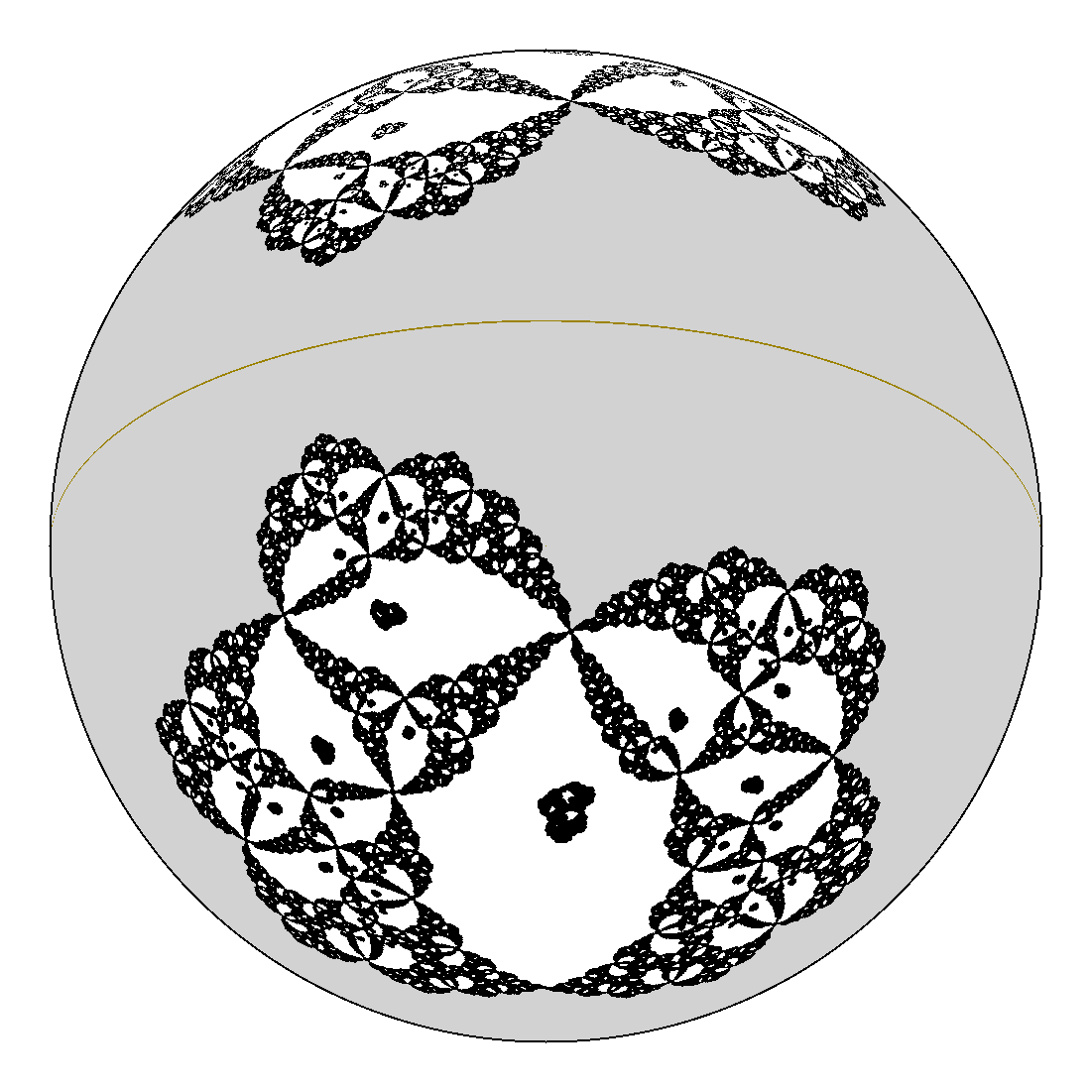}
  \caption{The Julia set of $f_{a,b}$ and its view on the Riemann sphere, where $a=-(3+\sqrt{13})/2$, $b=-1/2$ and $t=0.5865149\cdots$. The parameters are chosen such that $f_{a,b}$ has exactly two disjoint critical grand orbits and a Herman ring (the gray part) surrounding the unit circle (the brown curve) on which the rotation number is the golden mean $(\sqrt{5}-1)/2$.}
  \label{Fig_Julia-Herman-sphere}
\end{figure}

Let $\Lambda$ be a connected complex manifold. The family $f:\Lambda\times\EC\to\EC$ is called a \textit{holomorphic family} of rational maps \textit{parameterized} by $\Lambda$ if $f$ is holomorphic in two variables. In particular, $f_\lambda:\EC\to\EC$ is a rational map for each $\lambda\in\Lambda$. We say that a holomorphic family of rational maps $f:\Lambda\times\EC\to\EC$ has only one \textit{free} critical orbit, if there exists a critical point $c_\lambda$ of $f_\lambda$ such that for any $\widetilde{c}_\lambda\in\Crit(f_\lambda)\subt\{c_\lambda\}$, the forward orbit of $\widetilde{c}_\lambda$ is either finite or has a nonempty intersection\footnote{Two different critical orbits with nonempty intersection are regarded as one critical orbit essentially for convenience since they have the same grand orbit.} with the forward orbit of $c_\lambda$. As an immediate corollary of the Main Theorem, we have the following result:

\begin{thm}\label{thm-holo-fami}
The holomorphic family of rational maps having only one free critical orbit has no Herman rings.
\end{thm}

Now we can give several applications of the Main Theorem and Theorem \ref{thm-holo-fami}. For each $\omega\in\C\subt\{0\}$, the set of critical points of $f_\omega(z)=1+\omega/z^d$ is $\{0,\infty\}$, where $d\geq 2$. By definition, $f_\omega$ has only one free critical orbit $0\mapsto\infty\mapsto 1\mapsto 1+\omega\mapsto\cdots$. As an immediate corollary of Theorem \ref{thm-holo-fami}, we have following result\footnote{Bam\'{o}n and Bobenrieth's result can be obtained also by using Milnor's criterion on bicritical rational maps \cite{Mil00}.}.

\begin{cor}[{Bam\'{o}n and Bobenrieth}]
The rational maps $f_\omega(z)=1+\omega/z^d$ have no Herman rings, where $d\geq 2$ and $\omega\in\C\subt\{0\}$.
\end{cor}

Let $P$ be a non-constant complex polynomial with degree $d\geq 2$. The rational map $N_P:\EC\rightarrow\EC$ defined as $N_P(z)=z-P(z)/P'(z)$ is called the \emph{Newton's method} for $P$. The iterate of $N_P$ is used to find the roots of $P$. Except $\infty$, all the rest $d$ fixed points of $N_P$ are superattracting and they are exactly the roots of $P$. If $d=3$, then $N_P$ has at most one critical point in the Julia set. Therefore, we have following immediate corollary by the Main Theorem.

\begin{cor}[{Shishikura-Tan}]
The Newton's method for a cubic polynomial has no Herman rings.
\end{cor}

Actually, Shishikura proved that the Julia sets of the Newton's method for all polynomials are connected and hence they have no Herman rings \cite[Corollary II]{Shi09}, and Tan proved the cubic case in \cite[Proposition 2.6]{Tan97}.

Some other consequences of Theorem \ref{thm-holo-fami} have also been harvested. See \cite{YZ14} and \cite{XY16} for the study of the connectivity of the Julia sets of the holomorphic families with only one free critical orbit.

For the study of the necessary (or other) conditions of the existence of Herman rings of transcendental meromorphic functions, one can refer \cite{DF04}, \cite{FP12}, \cite{Nay16}, \cite{Zhe00} and the references therein.

\subsection{Organization of the article}

This article is organized as follows: In \S \ref{sec-proof}, we first show that a rational map with exactly one infinite critical orbit cannot have any fixed Herman rings. Then we prove that this rational map cannot have any Herman rings with period great than one and complete the proof of the Main Theorem.

In \S \ref{sec-sharp}, we do some calculations and analysis to find the parameters $a$, $b$ and $t$ such that the rational map $f_{a,b}$ defined in \eqref{equ-f-a-b} satisfies the properties in Theorem \ref{thm-sharp}.

In \S \ref{sec-example}, we show that there are lots of other families of rational maps satisfying the condition in Theorem \ref{thm-holo-fami}, including the generalized renormalization transformation family etc. Moreover, at the end of this section, we give another proof that the McMullen maps have no Herman rings by using Theorem \ref{thm-holo-fami}.

\vskip0.2cm
\noindent\textbf{Acknowledgements.} This work is supported by the NSFC under grant No.\,11401298 and the NSF of Jiangsu Province under grant No.\,BK20140587. The author would like to thank the referee for careful reading and useful suggestions.

\section{Proof of the Main Theorem}\label{sec-proof}

Let $f$ be a rational map with degree at least two. It is known that the boundary of the Herman ring has two connected components and each of them is an infinite set. If all the critical orbits in $J(f)$ are finite, then $J(f)$ has no Herman rings since every boundary point of a Herman ring is contained in the closure of some critical orbit in $J(f)$ (see \cite[Lemma 15.7]{Mil06}). Therefore, in the following, we always assume that $f$ has at least one infinite critical orbit in its Julia set.

\subsection{The fixed Herman rings}

We first discuss the case when the Herman rings are fixed. In this case, the proof of the Main Theorem is based on quasiconformal surgery. One can refer \cite{Shi87,Shi06} and \cite{BF14} for the basic ideas.

A continuous map $f:\EC\rightarrow\EC$ is called \emph{quasiregular} if it can be written as $f=\phi\circ g\circ \psi$, where $g$ is rational and $\phi,\psi$ are both quasiconformal mappings. For $s>0$ and $0<r<1$, let $\T_s:=\{z:|z|=s\}$ be the circle centered at the origin with radius $s$ and $\A_r:=\{z:r<|z|<1/r\}$ the annulus which is symmetric to the unit circle $\T_1$. For a Jordan curve $\gamma\subset\EC$ that does not pass through $\infty$, we use $\gamma_{\Ext}$ to denote the component of $\EC\subt \gamma$ containing $\infty$ and $\gamma_{\Int}$ the other. Let $A\subset\EC$ be an annulus with finite module such that $\infty\not\in\overline{A}$. Recall that the \textit{core curve} $\gamma_A$ of $A$ is defined as $\varphi^{-1}(\T_1)$, where $\varphi:A\rightarrow \A_r$ is a conformal isomorphism. Note that $\gamma_A$ is a smooth Jordan curve and it divides the Riemann sphere into two disjoint disks $A^{\Ext}:=(\gamma_A)_{\Ext}$ and $A^{\Int}:=(\gamma_A)_{\Int}$. We use $\partial_+A$ and $\partial_-A$ to denote the outer and inner boundary components of $A$ which are contained in $A^{\Ext}$ and $A^{\Int}$ respectively.

We say that a rational map $f$ has exactly one infinite critical orbit in $J(f)$ \textit{eventually} if there is a critical point $c$ in $J(f)$ such that the forward orbit of $c$ is infinite and for any other critical point $\widetilde{c}\in J(f)$, either the orbit $O_f(\widetilde{c})$ is finite or $\widetilde{c}$ has the same grand orbit as $c$, i.e. $f^{\circ i}(c)=f^{\circ j}(\widetilde{c})$ for some $i,j\in\N$.

\begin{thm}\label{thm-main2}
Let $f$ be a rational map with exactly one infinite critical orbit in $J(f)$ eventually. Then $f$ has no Herman rings with period one.
\end{thm}

\begin{proof}
We give the proof by contradiction. Suppose that $f$ has a fixed Herman ring $U$. Without loss of generality, we assume that $\infty\not\in\overline{U}$ since $\EC\subt\overline{U}\neq \emptyset$. By the definition of Herman ring, there exists a conformal map $\varphi:U\rightarrow\A_r$ and an irrational number $\theta\in(0,1)$ such that $\varphi\circ f\circ\varphi^{-1}(\zeta)=e^{2\pi\ii\theta}\zeta$. Then $f$ maps the outer and inner boundary components $\partial_+U$ and $\partial_-U$ to themselves respectively.
Let $\beta=\varphi^{-1}(\T_a)$ and $\eta=\varphi^{-1}(\T_b)$, where $r<a<b<1/r$. Then $\beta$ and $\eta$ are both analytic curves contained in $U$. Without loss of generality, we assume that $\beta\subset\eta_{\Int}$. Otherwise, one can use the map $1/\varphi$ to replace $\varphi$.

The conformal map $\varphi:\beta_{\Ext}\cap\eta_{\Int}\rightarrow\{\zeta:a<|\zeta|<b\}$ between these two annuli can be extended to a quasiconformal mapping defined from $\eta_{\Int}$ onto $\{\zeta:|\zeta|<b\}$, which we denote also by $\varphi$. Define
\[
g(z)=
\left\{
\begin{array}{ll}
f(z)  &~~~~~~~\text{if}~z\in \overline{\eta_{\Ext}}, \\
\varphi^{-1}(e^{2\pi\ii\theta}\varphi(z)) &~~~~~~\text{if}~z\in \eta_{\Int}.
\end{array}
\right.
\]
It is easy to see that $g$ is continuous on $\eta$ and hence it is quasiregular.
Let $\sigma_0$ be the standard complex structure on $\EC$. We use $\sigma$ to denote the $g$-invariant complex structure such that
\[
\sigma=
\left\{
\begin{array}{ll}
\varphi^*\sigma_0  &~~~~~~~\text{in~} \eta_{\Int}, \\
\sigma_0 &~~~~~~\text{outside of~} \bigcup_{n\geq 0}g^{-n}(\eta_{\Int}).
\end{array}
\right.
\]
By the measurable Riemann mapping theorem, there exists a quasiconformal mapping $h:\EC\to\EC$ integrating the almost complex structure $\sigma$ such that $h^*\sigma_0=\sigma$ and $h:(\EC,\sigma)\rightarrow (\EC,\sigma_0)$ is an analytic isomorphism. This means that $F:=h\circ g\circ h^{-1}:(\EC,\sigma_0)\to(\EC,\sigma_0)$ is a rational map.

Since $F(h(\overline{\eta_{\Int}}))=h(\overline{\eta_{\Int}})$, it means that $\{F^{\circ n}\}_{n\geq 0}$ is normal in $h(\eta_{\Int})$ and hence $h(\eta_{\Int})$ is contained in a Fatou component of $F$. Note that the conformal map $\varphi\circ h^{-1}:h(U\subt\overline{\eta_{\Int}})\rightarrow\{\zeta:b<|\zeta|<1/r\}$ conjugates $F$ to the rigid rotation $\zeta\mapsto e^{2\pi\ii\theta}\zeta$. This means that $h(U\cup \eta_{\Int})$ is contained in a Fatou component of $F$. We claim that $h(\partial_+U)$ is contained in the Julia set of $F$. Actually, let $z_0$ be any point in $\partial_+U$ and $V$ a neighborhood of $z_0$ such that $\infty\not\in\overline{V}$ and $\overline{V}\cap\overline{\eta_{\Int}}=\emptyset$. If $f^{\circ n}(z)\not\in\eta_\Int$ for any $z\in V$ and $n\geq 0$, then $\{F^{\circ n}\}_{n\in\N}=\{f^{\circ n}\}_{n\in\N}$ is not normal in $V$ since $z_0$ is contained in the Julia set $f$. Suppose that $f^{\circ n}(z)\in\eta_\Int$ for some $z\in V$ and $n\geq 0$. For any topological disk $\Omega\subset V$ which contains $z_0$ and $z$ with $\overline{\Omega}\subset V$, the sequence $\{F^{\circ n}\}_{n\in\N}$ is not normal in $\Omega$ since $\{F^{\circ n}\}_{n\in\N}$ has a subsequence converging locally uniformly to the identity in $h(\Omega\cap U)$. By the arbitrary of $z_0$ and $V$, the claim has been proved and hence $\deg(F)\geq 2$. This means that $h(U\cup \eta_{\Int})$ is a fixed Siegel disk of $F$.

The images of the critical points in the Fatou set of $f$ under $h$ are still in the Fatou set of $F$ (Note that the images of some critical points of $f$ under $h$ are not critical points of $F$ any more because of the surgery). By the assumption in the theorem, there exists at most one infinite critical orbit in $J(F)$ eventually. On the other hand, it was known that every boundary point of a Siegel disk is contained in the closure of some critical orbit in the Julia set (see \cite[Theorem 11.17]{Mil06}). Therefore, $F$ has exactly one infinite critical orbit in $J(F)$ eventually. In particular, there exists a critical point $\widetilde{c}$ of $F$ with infinite orbit in $J(F)$ such that $\widetilde{c}=h(c)$, where $c$ is a critical point of $f$ with infinite orbit in $J(f)$.

Note that the two boundary components of $U$ are contained in the closure of the forward orbit of $c$. There exists a smallest integer $k>0$ such that $f^{\circ k}(c)\in \eta_{\Int}$. Then we have
\begin{equation*}
F^{\circ k}(\widetilde{c})=F^{\circ k}\circ h(c)=h\circ g^{\circ k}(c)=h\circ f^{\circ k}(c).
\end{equation*}
However, $F^{\circ k}(\widetilde{c})$ lies in the Julia set of $F$ while $h\circ f^{\circ k}(c)\in h(\eta_{\Int})$ lies in a Siegel disk of $F$. This is a contradiction. The proof is complete.
\end{proof}

From the proof of Theorem \ref{thm-main2}, we have the following immediate corollary. See Figure \ref{Fig_cor-crit}.

\begin{cor}\label{cor-after-main}
$(1)$ Let $U$ be a fixed Herman ring of $f$ such that $\infty\not\in\overline{U}$. There exist two critical points $c_1\in U^{\Ext}\cap J(f)$ and $c_2\in U^{\Int}\cap J(f)$ with disjoint infinite critical orbits, such that $O_f(c_1)\subset U^{\Ext}$ and $O_f(c_2)\subset U^{\Int}$.

$(2)$ Let $U_1$ and $U_2$ be two fixed Herman rings of $f$ such that $\infty\not\in\overline{U}_1\cup \overline{U}_2$ and $U_1\subset U_2^\Int$. Then there exists a critical point $c\in U_1^{\Ext}\cap U_2^{\Int}\cap J(f)$ with infinite critical orbit, such that $O_f(c)\subset U_1^{\Ext}\cap U_2^{\Int}$.
\end{cor}

\begin{figure}[!htpb]
  \setlength{\unitlength}{1mm}
 \includegraphics[width=130mm]{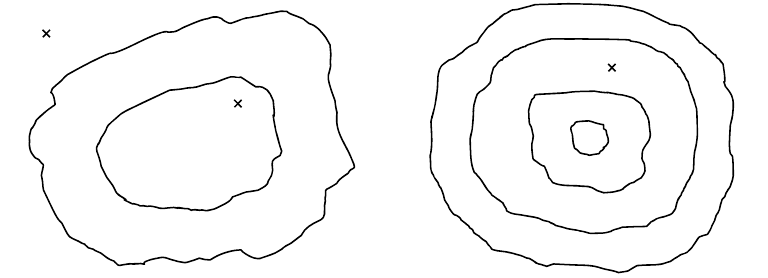}
  \put(-24,33){$c$}
  \put(-46,40){$U_2$}
  \put(-38,26){$U_1$}
  \put(-105,34){$U$}
  \put(-92,25){$c_2$}
  \put(-120,40){$c_1$}
  \caption{The positions of the critical points described in Corollary \ref{cor-after-main}.}
  \label{Fig_cor-crit}
\end{figure}
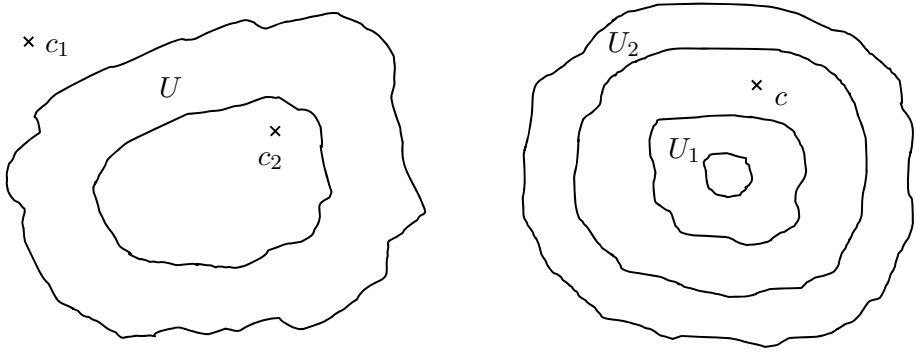

\subsection{The periodic Herman rings}

If $f$ has a Herman ring with period $p\geq 2$, we cannot use the surgery showed in Theorem \ref{thm-main2} directly. However, we can use Corollary \ref{cor-after-main} after considering the iteration $f^{\circ p}$. Note that the number of disjoint infinite critical orbits of $f^{\circ p}$ in $J(f^{\circ p})=J(f)$ is greater than one now.
Before stating of the non-existence of periodic Herman rings, we first need the following result.

\begin{lema}\label{lema-p}
For any $p\geq 1$, the collection of disjoint Jordan curves $\{\gamma_1,\cdots,\gamma_p\}$ divides $\EC$ into $p+1$ connected components.
\end{lema}

\begin{proof}
Obviously, lemma holds for $p=1$. Suppose that the collection of disjoint Jordan curves $\{\gamma_1,\cdots,\gamma_{p-1}\}$ divides $\EC$ into $p$ connected components. Let $\gamma_p$ be another Jordan curve satisfying $\gamma_p\cap\gamma_i=\emptyset$ for $1\leq i\leq p-1$. Then $\gamma_p$ is contained in a component of $\EC\subt\bigcup_{i=1}^{p-1}\gamma_i$ and $\gamma_p$ divides this component into exactly two components. Therefore, $\EC\subt\bigcup_{i=1}^p\gamma_i$ consists of $p+1$ components.
\end{proof}

\begin{thm}\label{thm-main3}
Let $f$ be a rational map with exactly one infinite critical orbit in $J(f)$ eventually. Then $f$ has no Herman rings with period $p\geq 2$.
\end{thm}

\begin{proof}
Suppose that $f$ has a cycle of periodic Herman rings $U_1\mapsto U_2\mapsto\cdots\mapsto U_p\mapsto U_1$ with period $p\geq 2$. Without loss of generality, considering the conjugacy of $f$ if necessary, we assume that $\infty\not\in\bigcup_{i=1}^p\overline{U}_i$ and $U_i\subset U_1^{\Int}$ for all $2\leq i\leq p$. Considering the iterate of $f$, we know that each $U_i$ is fixed by $f^{\circ p}$, where $1\leq i\leq p$.

Let $\gamma_i$ be the core curve of $U_i$, where $1\leq i\leq p$. Then $\{\gamma_1,\cdots,\gamma_p\}$ forms a family of disjoint Jordan curves and it divides $\EC$ into $p+1$ components, say $V_1$, $\cdots$, $V_{p+1}$. According to Corollary \ref{cor-after-main}, there are $p+1$ critical points of $f^{\circ p}$, say $c_1$, $\cdots$, $c_{p+1}$, such that each $c_i\in V_i\cap J(f^{\circ p})$ has an infinite critical orbit which is contained in $V_i$. Therefore, $f^{\circ p}$ has at least $p+1$ disjoint infinite critical orbits in $J(f^{\circ p})$. However, by the assumption that $f$ has exactly one infinite critical orbit in $J(f)$ eventually, its iterate $f^{\circ p}$ has exactly $p$ infinite critical orbits in $J(f^{\circ p})$ eventually. This is a contradiction and completes the proof.
\end{proof}

\begin{proof}[{Proof of the Main Theorem}]
This is an immediate corollary of Theorems \ref{thm-main2} and \ref{thm-main3}.
\end{proof}

\section{The sharpness of the Main Theorem}\label{sec-sharp}

We have proved in the last section that a rational map with a Herman ring must have at least two disjoint infinite critical orbits in its Julia set. In the present section, we construct an example to show that the Main Theorem is sharp: there exists a cubic rational map having exactly two critical grand orbits but also having a Herman ring, as stated in the introduction.

\thmsharpintro*

\begin{proof}
In order to simplify the calculation, we assume that $a$ is real and set $b=-1/2$. A direct calculation shows that
\begin{equation*}
f_{a,b}'(z)=e^{2\pi\ii t}\,\frac{(1+2z)(2-6a-2a^2+(11+a^2)z+(2-6a-2a^2)z^2)}{(2+z)^3(az-1)^2}.
\end{equation*}
We then set $2-6a-2a^2=0$, i.e. $a=-(3+\sqrt{13})/2$ since $|a|>1$. Note that $f_{a,b}$ is a Blaschke product whose dynamical behaviors are symmetric about the unit circle. It follows that $f_{a,b}$ has $4$ critical points $b=-1/2$, $1/\overline{b}=-2$, $0$ and $\infty$ with critical orbits $b\mapsto 0\mapsto f_{a,b}(0)\mapsto\cdots$ and $1/\overline{b}\mapsto \infty\mapsto f_{a,b}(\infty)\mapsto\cdots$. On the other hand, $f_{a,b}$ has two zeros $b$ (counted with multiplicity) and one pole $1/\overline{a}$ in the unit disk. It follows from the argument principle that the restriction of $f_{a,b}$ on the unit circle is a real analytic diffeomorphism for any $t\in(0,1)$.

For any Diophantine number $\theta\in(0,1)$, there exists a unique $t\in(0,1)$ such that the rotation number of $f_{a,b}$ on the unit circle is $\theta$ (see \cite[p.\,32]{dMvS93}). By Herman-Yoccoz's linearization theorem \cite[Theorem 1.4]{Yoc02}, $f_{a,b}$ can be complex analytically conjugated to the rigid rotation $\zeta\mapsto e^{2\pi\ii\theta}\zeta$ in a neighborhood of the unit circle. Note that $0$ and $\infty$ are both critical points that cannot contained in the rotation domains. This means that $f_{a,b}$ has a fixed Herman ring surrounding the unit circle. This completes the proof of Theorem \ref{thm-sharp}.
\end{proof}

Actually, the Blaschke product $f_{a,b}$ in Theorem \ref{thm-sharp} can be regarded as a map obtained by quasiconfomal surgery. Indeed, let
\begin{equation}
f_c(z)=e^{2\pi\ii\theta}\frac{z^2}{z-c}+2c, \text{ where } c\neq 0.
\end{equation}
If $\theta$ is of bounded type, then $f_c$ has a Siegel disk $\Delta_\infty$ centered at the infinity whose boundary contains at least one critical point\footnote{Yampolsky and Zakeri proved that the boundary of the bounded type Siegel disk of any quadratic rational map contains at least one critical point \cite{YZ01} and Zhang generalized the result to all rational maps with degree at least two in \cite{Zha11} by quasiconformal surgery. Graczyk and \'{S}wi\c{a}tek proved a more general result in \cite{GS03} by using Schwarzian derivative: if an analytic function has a Siegel disk properly contained in the domain of holomorphy and the rotation number is of bounded type, then the corresponding boundary of the Siegel disk contains a critical point.} (see \cite{YZ01}, \cite{Zha11} or \cite{GS03}). Indeed, since $f_c$ has exactly two critical points $0$ and $2c$, and $f_c(0)=2c$, thus $\partial\Delta_\infty$ contains at least the critical point $2c$. It means that $f_c$ is a quadratic rational map having exactly one critical grand orbit and having also a Siegel disk centered at the infinity.

By a standard Siegel-to-Herman surgery under the symmetric normalization (see \cite[\S 9]{Shi87}, \cite[\S 7]{BF14} and \cite[\S 4]{Zak10}, the inverse process, Herman-to-Siegel surgery, has appeared in the proof of Theorem \ref{thm-main2}), one can obtain a cubic Blaschke product $\widetilde{f}$ such that $\widetilde{f}$ has exactly two disjoint critical grand orbits and has also a Herman ring surrounding the unit circle. After normalizing the positions of the critical points such that $\widetilde{f}$ maps a critical point to the origin, the map $\widetilde{f}$ must have the form in Theorem \ref{thm-sharp}. See Figure \ref{Fig_Julia-Siegel-Herman}.

\begin{figure}[!htpb]
  \setlength{\unitlength}{1mm}
  \includegraphics[width=70mm]{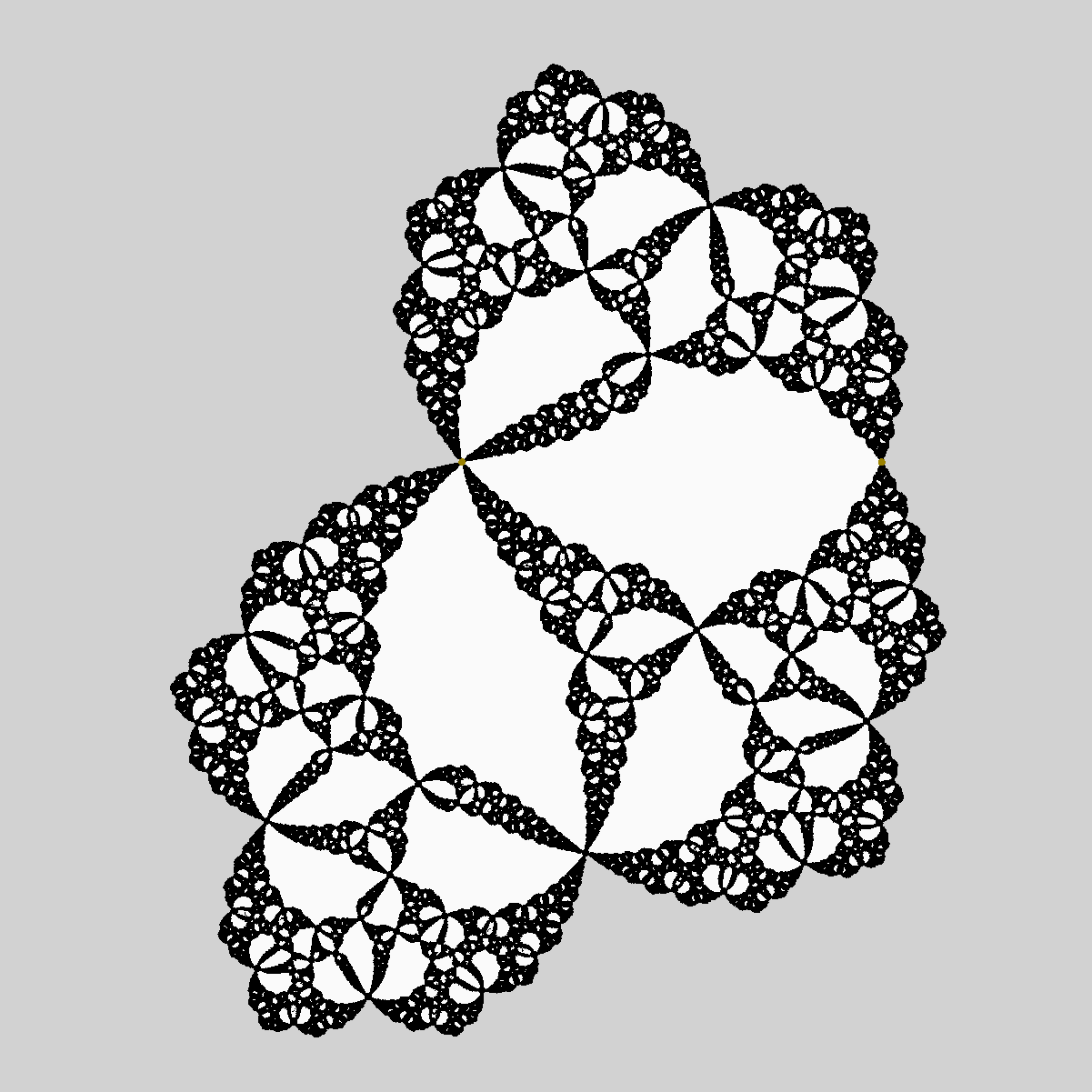}\quad
  \includegraphics[width=70mm]{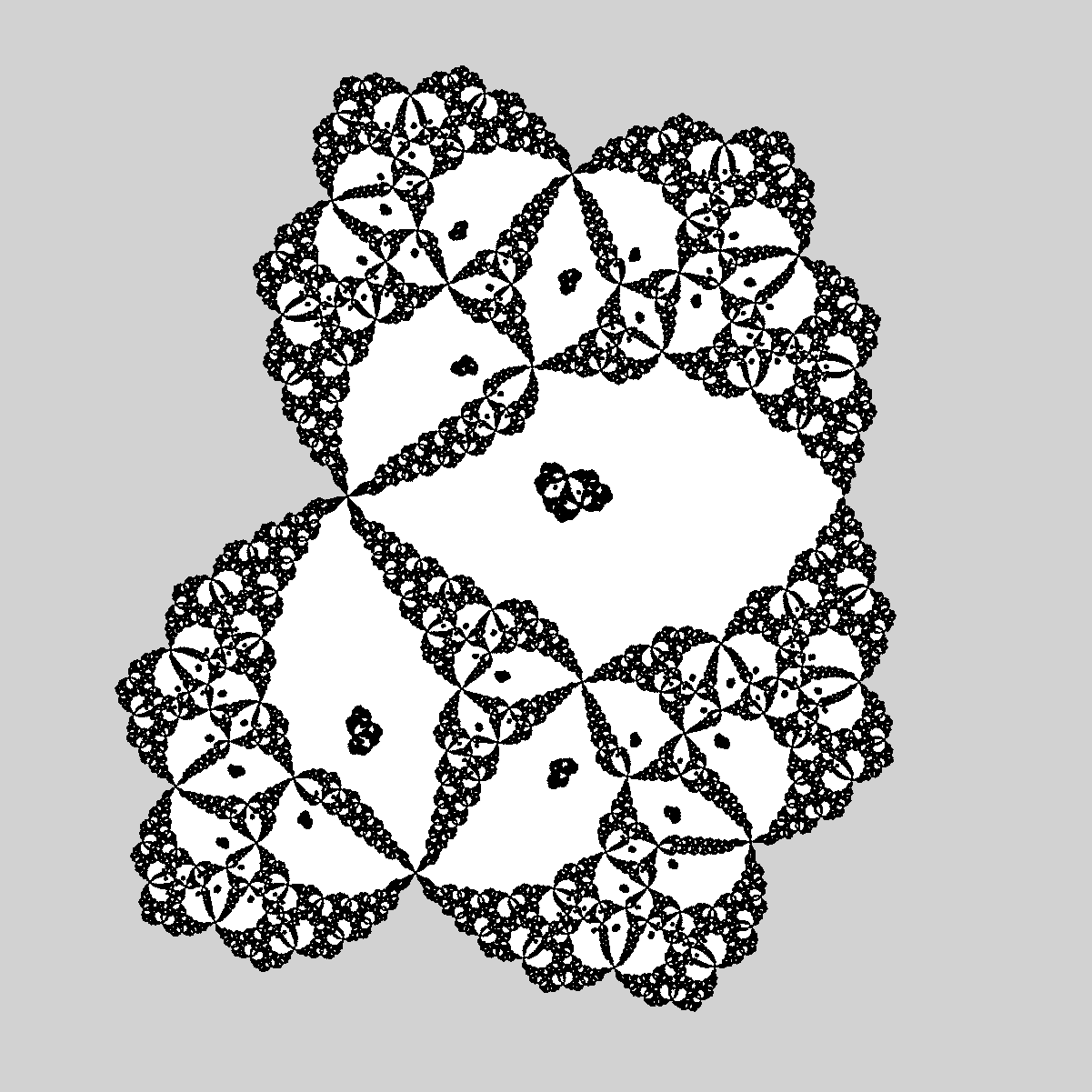}
  \caption{The Julia sets of $f_c$ and $f_{a,b}$, where $c=1/2$ such that $f_c$ has a fixed Siegel disk (the gray part on the left) and $a=-(3+\sqrt{13})/2$, $b=-1/2$ and $t=0.5865149\cdots$ such that $f_{a,b}$ has a fixed Herman ring (the gray part on the right). Both of the rotation numbers in the rotation domains are chosen to be the golden mean. The black parts on the right are subsets of the Julia set of $f_{a,b}$ which are contained in the unit disk. Compare Figure \ref{Fig_Julia-Herman-sphere}.}
  \label{Fig_Julia-Siegel-Herman}
\end{figure}

\section{Applications: rational maps having no Herman rings}\label{sec-example}

In this section, we give some examples such that one can apply the Main Theorem or Theorem \ref{thm-holo-fami}. In complex dynamics, one may want to study the rational maps in family. Usually, the family has only one parameter such that the family has only one free critical orbit (or has only one essentially, such as McMullen maps, see \S\ref{subsec-mcm}) since it is convenient to define the corresponding connected locus. Recall that a holomorphic family of rational maps $f_\lambda:\EC\to\EC$ with $\lambda\in\Lambda$ is said to have only one \textit{free} critical orbit, if there exists a critical point $c_\lambda$ of $f_\lambda$ such that for any $\widetilde{c}_\lambda\in\Crit(f_\lambda)\subt\{c_\lambda\}$, the forward orbit of $\widetilde{c}_\lambda$ is either finite or has a nonempty intersection with the forward orbit of $c_\lambda$.

\subsection{The generalized renormalization transformations have no Herman rings}

Let
\begin{equation*}\label{Umn}
S_\lambda(z) =\left(\frac{(z+\lambda-1)^d+(\lambda-1)(z-1)^d}{(z+\lambda-1)^d-(z-1)^d}\right)^d,
\end{equation*}
where $d\geq 2$ is an integer and $\lambda\neq 0$ is a complex parameter. This family of rational maps are actually the renormalization transformation of the generalized diamond hierarchical Potts model \cite{Qia15}. The special case named \emph{standard diamond lattice} ($d=2$) was first studied in \cite{DSI83}.

\begin{prop}
The renormalization transformation $S_\lambda$ has no Herman rings.
\end{prop}

\begin{proof}
For every $\lambda\in\C\subt\{0\}$, we have $S_\lambda=T_\lambda\circ T_\lambda$, where
\begin{equation*}
T_\lambda(z)=\Big(\frac{z+\lambda-1}{z-1}\Big)^d.
\end{equation*}
A directly calculation shows that the set of all critical points of $T_\lambda$ is $\Crit (T_\lambda)=\{1,1-\lambda\}$ and both with multiplicity $d-1$.
Under the iteration of $T_\lambda$, we have the following critical orbits:
\begin{equation*}
1\mapsto\infty\mapsto 1\mapsto\infty\mapsto \cdots \text{~~and~~}
1-\lambda\mapsto 0\mapsto (1-\lambda)^d\mapsto\cdots.
\end{equation*}
From the first orbit, we know that $1$ lies in the Fatou set of $T_\lambda$. This means that there exists at most one infinite critical orbit in $J(T_\lambda)=J(S_\lambda)$. By Theorem \ref{thm-holo-fami}, it follows that $T_\lambda$ and hence $S_\lambda$ has no Herman rings.
\end{proof}

\subsection{The McMullen maps have no Herman rings}\label{subsec-mcm}

In this subsection, we give another proof of the nonexistence of Herman rings of McMullen maps by using Theorem \ref{thm-holo-fami}.
For each $m,\ell\geq 2$, let
\begin{equation}\label{eq-McMullen-2}
F_\lambda(z)=z^m+\lambda/z^\ell
\end{equation}
be the McMullen family with parameter $\lambda\in\C\subt\{0\}$.

\begin{prop}[{Xiao and Qiu, \cite{XQ10}}]\label{thm-McM}
The McMullen maps have no Herman rings.
\end{prop}

\begin{proof}
For any $\lambda\neq 0$, let $\phi(z)=z^{m+\ell}/\lambda$. We define a new map
\begin{equation*}
G_\lambda(w)=\lambda^{m-1} w^m\Big(1+\frac{1}{w}\Big)^{m+\ell}.
\end{equation*}
Then one can check easily that $G_\lambda\circ \phi=\phi\circ F_\lambda$, i.e. $F_\lambda$ is semi-conjugated to $G_\lambda$ by $\phi$. A direct calculation shows that $G_\lambda$ has exactly $3$ critical points $0$, $\infty$ and $\ell/m$ (counted without multiplicity). Note that $0$ and $\infty$ are both mapped to $\infty$ and $\infty$ is a superattracting fixed point of $G_\lambda$. This means that $G_\lambda$ has at most one critical orbit in the Julia set of $G_\lambda$. By Theorem \ref{thm-holo-fami}, $G_\lambda$ has no Herman rings.

If the critical point $\ell/m$ is contained in the Fatou set of $G_\lambda$, then all the critical points of $F_\lambda$ are contained in the Fatou set of $F_\lambda$ and hence $F_\lambda$ has no Herman rings. Suppose that $\ell/m$ is contained in the Julia set of $G_\lambda$. Then besides the superattracting basin centered at $\infty$, the periodic Fatou components of $G_\lambda$ can only be Siegel disks. This means that all the Fatou components of $G_\lambda$ are either simply connected or infinitely connected. Since $F_\lambda$ is semi-conjugated to $G_\lambda$ by $\phi$, it follows that each Fatou component of $F_\lambda$ is either simply connected or infinitely connected. Therefore, $F_\lambda$ has no Herman rings.
\end{proof}


\end{document}